\documentclass[a4paper,10pt]{amsart}
\usepackage{enumerate}
\usepackage{amsthm}
\usepackage{mathrsfs}
\usepackage{color}
\usepackage[firstinits=true,backend=bibtex,maxbibnames=9,isbn=false,url=false,doi=false]{biblatex}
\bibliography{refs}

\usepackage[utf8]{inputenc}
\usepackage[T1]{fontenc}
\usepackage[margin=1.5in]{geometry}

\newtheorem{theorem}{Theorem}[section]
\newtheorem{lemma}[theorem]{Lemma}
\newtheorem{res}[theorem]{Main result}

\newtheorem{corollary}[theorem]{Corollary}
\newtheorem{property}[theorem]{Property}
\newtheorem{fact}[theorem]{Fact}

\theoremstyle{definition}

\theoremstyle{remark}

\newcommand{\E} {\ensuremath {\mathbb{E}}}

\newcommand{\N} {\ensuremath {\mathbb{N}}}
\newcommand{\R} {\ensuremath {\mathbb{R}}}

\newcommand{\I} {\ensuremath {\mathbb{I}}}
\newcommand{\F} {\ensuremath {\mathscr{F}}}

\newcommand{\Nn} {\ensuremath {\mathscr{Nn}}}

\newcommand{\A} {\ensuremath {\mathscr{A}}}
\newcommand{\X} {\ensuremath {\mathscr{X}}}

\newcommand{\B} {\ensuremath {\mathscr{B}}}
\newcommand{\M} {\ensuremath {\mathscr{M}}}

\newcommand{\D} {\ensuremath {\mathscr{D}}}

\newcommand{\Ci} {\ensuremath {\mathscr{C}}}
\newcommand{\mo} {\ensuremath {\mathscr{P}}}
\newcommand{\No} {\ensuremath {\mathcal{N}}}
\newcommand{\Wh} {\ensuremath {\mathscr{W}}}

\title[Asymptotic equivalence for density estimation experiments]{Asymptotic equivalence for density estimation and Gaussian white noise: an extension}
\author{Ester ~Mariucci}
\address{\it Laboratoire LJK, Universit\'e Joseph Fourier UMR 5224\\
    \it 51, Rue des Math\'ematiques, Campus de Saint Martin d'H\`eres\\
        \it BP 53 38041 Grenoble Cedex 09}

\email{Ester.Mariucci@imag.fr}
\date{\today}
\begin{document}
\begin{abstract}
 The aim of this paper is to present an extension of the well-known asymptotic equivalence between density estimation experiments and a Gaussian white noise model. Our extension consists in enlarging the nonparametric class of the admissible densities. More precisely, we propose a way to allow densities defined on any subinterval of $\R$, and also some discontinuous or unbounded densities are considered (so long as the discontinuity and unboundedness patterns are somehow known a priori). The concept of equivalence that we shall adopt is in the sense of the Le Cam distance between statistical models. The results are constructive: all the asymptotic equivalences are established by constructing explicit Markov kernels.
\end{abstract}

\maketitle
\section{Introduction}
 When looking for asymptotic results for some statistical model it is often useful to profit from a global asymptotic equivalence, in the Le Cam sense, in order to be allowed to work in a simpler but equivalent model. Indeed, proving an asymptotic equivalence result means that one can transfer asymptotic risk bounds for any inference problem from one model to the other, at least for bounded loss functions. Roughly speaking, saying that two models, $\mo_1$ and $\mo_2$, are equivalent means that they contain the same amount of information about the parameter that we are interested in. For the basic concepts and a detailed description of the notion of asymptotic equivalence, we refer to \cite{lecam, LC2000}. A short review of this topic will be given in Appendix. 
 
 In recent years, numerous papers have been published on the subject of nonparametric asymptotic equivalence. For a non exhaustive list of the main ones among them, see, for example, the introduction in \cite{esterlevy}. In this paper, we will focus on nonparametric density estimation experiments.
 

The seminal paper in this subject is due to Nussbaum \cite{N96}. There, the asymptotic equivalence between an experiment given by $n$ observations of a density $f$ on $[0,1]$ and a Gaussian white noise model:
$$dy_t=\sqrt {f(t)}dt+\frac{1}{2\sqrt n}dW_t,\quad t\in[0,1],$$
was established. Over the years several generalizations of this result have been proposed such as \cite{BC04,j03,cmultinomial}.
In \cite{BC04}, Brown et al. obtained the global asymptotic equivalence between a Poisson process with variable intensity and a Gaussian white noise experiment with drift problem. Via Poissonization, this result was also extended to density estimation models. In \cite{j03} Jähnisch and Nussbaum proved the global asymptotic equivalence between a nonparametric model associated with the observation of independent but not identically distributed random variables on the unit interval and a bivariate Gaussian white noise model.
More closely related to our work is the result of Carter in \cite{cmultinomial}. In that paper, he proposed a new approach to establish the same normal approximations to density estimations experiments as in \cite{N96}. While the result in \cite{N96} is obtained by means of Poissonization, in \cite{cmultinomial} the key step is to connect the density estimation problem to a multinomial experiment and to simplify the latter with a multivariate normal experiment. 

The purpose of the present work is to generalize \cite{N96} and \cite{cmultinomial}.
More precisely, the density estimation experiments that we consider consist of $n$ independent observations $(Y_i)_{i=1}^n$ defined on a interval $I\subseteq \R$ from some unknown distribution $P_f^g$ having density (with respect to the Lebesgue measure on $I$) $\frac{dP_f^g}{dx}(x)=f(x)g(x).$ 
In particular, we do not require $I \subseteq \mathbb R$ to be bounded as is generally done in the existing literature.
The function $g$ is supposed to be known whereas $f$ is unknown and belongs to a certain nonparametric functional class $\F$. Formally, the statistical model we consider is 
\begin{equation}\label{eq:f}
 \mo_n^g=\big(\R^n,\B(\R^n),\{P_f^g:f\in\F\}\big).
 \end{equation}
The exact assumptions on $f$ and $g$ will be specified in Section \ref{sec:parameter}. Here, let us only stress the fact that $f$ has to be bounded away from zero and infinity and sufficiently regular, whereas $g$ can be both unbounded and discontinuous. The advantage with respect to the earlier works is that this framework allows us to treat densities of the form $h=fg$ not necessarily bounded nor smooth. See Section \ref{sec:discussion} for a discussion about the hypotheses. 
 
Finally, let us introduce the Gaussian white noise model. For that, let us denote by $(C,\Ci)$ the space of continuous mappings from $I$ into $\R$ endowed with its standard filtration and by $\mathbb W_f^g$ the law induced on $(C,\Ci)$ by a stochastic process satisfying:
\begin{equation}\label{eq:Y}
 dY_t=\sqrt{f(t)g(t)}dt+\frac{dW_t}{2\sqrt{n}},\quad t\in I,
\end{equation}
where $(W_t)_{t\in\R}$ is a Brownian motion on $\R$ conditional on $W_0=0$.
Then we set
\begin{equation}\label{eq:wh}
 \Wh_n^g=\big(C,\Ci,\{\mathbb W_f^g:f\in\F\}\big).
\end{equation}
Let $\Delta$ be the Le Cam pseudo-distance between statistical models having the same parameter space. For the convenience of the reader a formal definition is given in Section \ref{sec:lecam}. Our main result is then as follows (see Theorem \ref{teo1} for the precise statement):
\begin{res}
Let $I$ be a possibly infinite subinterval of $\mathbb{R}$ and let $\F$ consist of functions bounded away from $0$ and $\infty$, satisfying the regularity assumptions stated in Section \ref{sec:parameter}. Then, we have
 \begin{equation}\label{teointro}
\lim_{n\to\infty}\Delta(\mo_n^g,\Wh_n^g)=0.  
 \end{equation}
\end{res}
In some special cases an explicit upper bound for the rate of convergence in \eqref{teointro} is available; see, e.g. Corollary \ref{cor1}.
The structure of the proof follows Carter's in \cite{cmultinomial}, but we detach from it on several aspects. The basic idea is to use his multinomial-multivariate normal approximation, but some technical points have to be taken into account. One of these is that $I$ may be infinite, so that, in particular, the subintervals $J_i$ in which it is partitioned cannot be of equal length. We choose intervals $J_i$ of varying length, according to the quantiles of $\nu_0$, the measure having density $g$ with respect to Lebesgue. This kind of partitions was already considered in \cite{esterlevy}.

The paper is organized as follows. Section \ref{sec:parameter} fixes the assumptions on the parameter space $\F$. Section \ref{sec:mainresult} contains the statement of the main results and a discussion while Section \ref{sec:proofs} is devoted to the proofs. The paper includes an Appendix recalling the definition and some useful properties of the Le Cam distance. 
\section{The parameter space}\label{sec:parameter}
Fix a finite measure $\nu_0$ on a possibly infinite interval $I\subset \R$, admitting a density $g$ with respect to Lebesgue. The class of functions $\F$ will be considered as a class of probability densities with respect to $\nu_0$, i.e. $\int_I f(x)g(x)dx=1$. For each $f\in\F$, let $\nu$ (resp. $\hat \nu_m$) be the measure having $f$ (resp. $\hat f_m$) as a density with respect to $\nu_0$ where, for every $f\in\F$, $\hat f_m(x)$ is defined as follows. Given a positive integer $m$, let $J_1= I\cap (-\infty, v_1]$, $J_j:=(v_{j},v_{j+1}]$ for $j=1,\dots,m-1$ and $J_m= I\cap (v_m,\infty)$ where the $v_j$'s are the quantiles for $\nu_0$, i.e.
\begin{equation}\label{eq:Jj}
\mu_n:=\nu_0(J_j)=\frac{\nu_0(I)}{m},\quad \forall j=1,\dots,m.
\end{equation}
Define  $x_j^*:=\frac{\int_{J_j}x\nu_0(dx)}{\mu_n}$ and
\begin{equation}\label{eq:hatf}
\hat f_m(x):=
 \begin{cases}
\frac{\nu(J_1)}{\mu_n} & \textnormal{if } x\in  I\cap (-\infty, x_1^*],\\
\frac{1}{x_{j+1}^*-x_j^*}\bigg[\frac{\nu(J_{j+1})}{\mu_n}(x-x_j^*)+\frac{\nu(J_{j})}{\mu_n}(x_{j+1}^*-x)\bigg] & \textnormal{if } x\in (x_j^*,x_{j+1}^*] \quad j = 1,\dots,m-1,\\
\frac{\nu(J_m)}{\mu_n} & \textnormal{if } x\in I\cap (x_m^*,\infty).
 \end{cases}
\end{equation}

We now explain the assumptions we will need to make on the parameter $f$. 
We require that:
\begin{enumerate}[(H1)]
 \item  There exist constants $\kappa, M >0$ such that $\kappa \leq f(y)\leq M$, for all $y\in I$ and $f\in \F$.
\end{enumerate}

The $m$ introduced above will be considered as a function of $n$, $m = m_n$. We can thus consider $\widehat{\sqrt {f}}_m$, the linear interpolation of $\sqrt{f}$ constructed as $\hat f_m$ above and introduce the quantities:
  \begin{align*}
   H_m^2(f)&:= \int_I\Big(\sqrt{f(x)}-\sqrt{\hat f_m(x)}\Big)^2 \nu_0(dx),\\
   A_m^2(f)&:= \int_I\Big(\widehat{\sqrt {f}}_m(y)-\sqrt{f(y)}\Big)^2\nu_0(dy),\\
   B_m^2(f)&:= \sum_{j=1}^m\bigg(\int_{J_j}\frac{\sqrt{f(y)}}{\sqrt{\nu_0(J_j)}}\nu_0(dy)-\sqrt{\nu(J_j)}\bigg)^2.
  \end{align*}
We will assume the existence of a sequence of discretizations $m = m_n$ such that:
\begin{enumerate}[(C1)]
 \item \label{cond:ABC}$\lim\limits_{n \to \infty}n\sup\limits_{f \in\F} \big(H_m^2(f)+A_m^2(f)+B_m^2(f)\big)=0$. 
\end{enumerate}  

\section{Main results and discussion}\label{sec:mainresult}
Using the notation introduced in Section \ref{sec:parameter}, we now state our main result in terms of the models $\mo_n^g$ and $\Wh_n^g$ defined in \eqref{eq:f} and \eqref{eq:wh}, respectively. 
\begin{theorem}\label{teo1}
Let $\nu_0$ be a finite measure on an (possibly infinite) interval $I\subset \R$ having density $g$ with respect to Lebesgue. Suppose that there exists a sequence $m = m_n$ such that
 every $f \in \F$ satisfies conditions (H1) and (C1). 
Then, for $n$ big enough we have:
$$\Delta(\mo_n^g, \Wh_n^g) = O\bigg(\sqrt{n}\sup_{f\in \F}\Big(A_m(f)+B_m(f)+H_m(f)\Big)+\frac{m\ln m}{\sqrt n}\bigg).$$
\end{theorem}

\begin{corollary}\label{cor1}
 Let $I$ be a compact subset of $\R$. For fixed $\gamma\in (0,1]$ and $K,\kappa, M$ strictly positive constants, consider the functional class
 $$\F_{(\gamma,K,\kappa,M)}=\Big\{f\in C^1(I): \varepsilon\leq f(x)\leq M, \ |f'(x)-f'(y)|\leq K|x-y|^{\gamma},\ \forall x,y\in I\Big\}.$$
 Suppose $\F\subset \F_{(\gamma,K,\kappa,M)}$. Then
 \begin{equation*}\label{eq:cor}
\Delta(\mo_n^g, \Wh_n^g)=O\Big(\sqrt n \big(\ell_m^{\gamma+1}+\sqrt \mu_n\ell_m\big)+\frac{m\ln m}{\sqrt n}\Big),
 \end{equation*}
 where $\ell_m=\max_{i=1,\dots,m}|v_j-v_{j-1}|$, with the $v_i$'s defined as in Section \ref{sec:parameter}.
 \end{corollary}
\subsection{Existing literature and discussion}\label{sec:discussion}
As it has already been highlighted in the introduction, our result is a generalization of those in \cite{N96} and \cite{cmultinomial}. In order to discuss the link between our work and the previous ones, we recall the results contained in these papers.
\begin{itemize}
 \item \emph{Asymptotic equivalence of density estimation and Gaussian white noise}, \cite{N96}.
 
 In this paper Nussbaum establishes a global asymptotic equivalence between the problem of density estimation from an i.i.d. sample and a Gaussian white noise model. More precisely, let $(Y_i)_{i=1}^n$ be i.i.d. random variables with density $f$ on $[0,1]$ with respect to the Lebesgue measure. 
 The densities $f$ are the unknown parameters and they are supposed to belong to a certain nonparametric class $\F$ subject to a Hölder restriction: $|f(x)-f(y)|\leq C|x-y|^\alpha$ with $\alpha>\frac{1}{2}$ and a positivity restriction: $f(x)\geq \varepsilon >0$. Let us denote by $\mo_{1,n}$ the statistical model associated with the observation of the $Y_i$'s. Furthermore, let $\mo_{2,n}$ be the experiment in which one observes a stochastic process $(Y_t)_{t\in[0,1]}$ such that
$$dY_t=\sqrt{f(t)}dt+\frac{1}{2\sqrt n}dW_t,\quad t\in[0,1]$$
where $(W_t)_{t\in[0,1]}$ is a standard Brownian motion.
Then the main result in \cite{N96} is that $\Delta(\mo_{1,n},\mo_{2,n})\to 0$ as $n\to\infty$.

This is done by first showing that the result holds for certain subsets $\F_n(f_0)$ of the class $\F$ described above. Then it is shown that one can estimate the $f_0$ rapidly enough to fit the various pieces together. 
Without entering into any detail, let us just mention that the key steps are a Poissonization technique and the use of a functional KMT inequality.

\item \emph{Deficiency distance between multinomial and multivariate normal experiments,} \cite{cmultinomial}. 

In this paper Carter establishes a global asymptotic equivalence between a density estimation model and a Gaussian white noise model by bounding the Le Cam distance between multinomial and multivariate normal random variables. More precisely, let us denote by $\mathcal M(n,\theta)$ the multinomial distribution, where $\theta:=(\theta_1,\dots,\theta_m)$. Denote the covariance matrix $nV_\theta$: Its $(i,j)$th element equals to $n\theta_i(1-\theta_i)\delta_{i,j}-n\theta_i\theta_j$. 

The main result is an upper bound for the Le Cam distance $\Delta(\mathcal M,\mathcal N)$ between the models $\mathcal M:=\{\mathcal M(n,\theta):\theta\in\Theta\}$ and $\mathcal N:=\{\Nn(n\theta,nV_\theta):\theta\in\Theta\}$, under some regularity assumptions on $\Theta$. In particular, Carter proves that
$$\Delta(\mathcal M,\mathcal N)\leq C'_\Theta\frac{m\ln m}{\sqrt n} \quad \textnormal{ provided } \ \sup_{\theta\in\Theta}\frac{\max_i \theta_i}{\min_i \theta_i}\leq C_\Theta<\infty,$$
for a constant $C'_\Theta$ that depends only on $C_\Theta$. 
From this inequality Carter can recover most the same results as Nussbaum \cite{N96} under stronger regularity assumptions on $\mathscr{F}$: $\F$ 
is a class of smooth, differentiable densities $f$ on the interval $[0,1]$ such that there exist strictly positive constants $\varepsilon,M,\gamma$ such that $\varepsilon\leq f\leq M$ and
$$|f'(x)-f'(y)|\leq M |x-y|^{\gamma}, \quad \textnormal{for all } x,y\in [0,1].$$
Let us briefly explain how one can use a bound on the distance between multinomial and multivariate normal variables to make assertions about density estimation experiments. The idea is to see the multinomial experiment as the result of grouping independent observations from a continuous density into subsets. Using the square root as a variance-stabilizing transformation, these multinomial variables can be asymptotically approximated by normal variables with constant variances. These normal variables, in turn, are approximations to the increments of the Brownian motion processes over the sets in the partition. 
\end{itemize}
Our work can be seen as a generalization of the previously cited works: To see that it is enough to take $g(x)=\I_{[0,1]}(x)$ and apply Corollary \ref{cor1}. However, it differs from Nussbaum and Carter's results in several aspects. First of all, we do not need to ask the random variables to be defined on $[0,1]$, allowing the observations to be defined on a possibly infinite interval $I$ of $\mathbb{R}$. Secondly, in our setting the positivity restriction on the densities can be removed. Indeed, as a parametric example, we can consider truncated Gamma distributions on $[0,L]$, that is distributions having a density $h$ with respect to the Lebesgue measure: 
$$h(x)=\frac{\exp(-\theta x)\theta^n x^{n-1}}{\int_0^L \exp(-\theta y)\theta^n y^{n-1}dy}\I_{[0,L]}(x).$$
We can apply Theorem \ref{teo1}, taking $\F=\{f_\theta:\theta\in\R_{>0}\}$ and
$$f_\theta(x)=\frac{\exp(-\theta x)\theta^n}{\int_0^L \exp(-\theta y)\theta^n y^{n-1}dy}\I_{[0,L]}(x),\quad g(x)=x^{n-1}.$$

More generally, density functions $h$ that can be written in form of a product are commonly used in statistics. Again, one could cite as a simple case the problem of a parametric estimation for a Weibull density, see, e.g. \cite{gardes, diebolt}. 
Generally speaking, the present work can be useful whenever the random variables $Y_i$'s do not admit a smooth density $h$ with respect to Lebesgue, but nevertheless one has some informations on the discontinuity structure, namely one knows $g$ in the decomposition $h(x) = f(x)g(x)$.
\section{Proofs}\label{sec:proofs}
\subsection{Proof of Theorem \ref{teo1}}

We will proceed in four steps.

Step 1. By means of Facts \ref{h} and \ref{hp}, we get

\begin{align*}
 \bigg\|\bigotimes_{i=1}^n P_f^g-\bigotimes_{i=1}^n P_{\hat f_m}^g\bigg\|_{TV}&\leq H\bigg(\bigotimes_{i=1}^n P_f^g,\bigotimes_{i=1}^n P_{\hat f_m}^g\bigg)
                                       \leq\sqrt{n H^2\big(P_{f}^g,P_{\hat f_m}^g\big)}.
\end{align*}
Hence, denoting by $\hat\mo_n^g$ the statistical model associated with the family of probabilities $\big\{P_{\hat f_m}^g;f\in\F\big\}$:
\begin{equation}\label{eq:fhatf}
\Delta(\mo_n^g,\hat\mo_n^g)\leq \sqrt{n\int_I\bigg(\sqrt {f(x)}-\sqrt{\hat f_m(x)}\bigg)^2g(x)dx}.
\end{equation}

Step 2. Following the same approach as in \cite{cmultinomial}, we introduce an auxiliary multinomial experiment to get closer to a normal one representing the increments of $(Y_t)_{t\in I}$ defined as in \eqref{eq:Y}. The multinomial experiment is linked with the density estimation model in the following way: Let $\hat{Y}_i$ be a set of i.i.d. random variables with density $\hat{f}_m g$ with respect to Lebesgue and define the multinomial experiment by grouping their observations into subsets. More precisely, let us introduce the random variables:
$$Z_i=\sum_{j=1}^n\I_{J_i}(\hat Y_j),\ i=1,\dots,m.$$
Observe that the law of the vector $(Z_1,\dots,Z_m)$ is multinomial $\mathcal M(n;\gamma_1,\dots,\gamma_m)$ where
$$\gamma_i=\int_{J_i}f(x)g(x)dx,\quad i=1,\dots,m;$$
here we have used the fact that $\int_{J_i}f(x)g(x)dx=\int_{J_i}\hat f_m(x)g(x)dx$.
Let us denote by $\mathcal M_m$ the statistical model associated with the observation of $(Z_1,\dots,Z_m)$. Clearly $\delta(\hat\mo_n^g,\mathcal M_m)=0$. Indeed, $\mathcal M_m$ is the image experiment by the random variable $S:I^n\to\{1,\dots,n\}^{m}$ defined as 
$$S(x_1,\dots,x_n)=\Big(\#\big\{j: x_j\in J_1\big\};\dots;\#\big\{j: x_j\in J_m\big\}\Big),$$
where $\# A$ denotes the cardinal of the set $A$. To conclude the second step we now prove that the multinomial experiment is as informative as $\hat\mo_n^g$:

\begin{lemma}\label{lemma:multinomial}
 $$\delta(\mathcal M_m,\hat\mo_n^g) = 0.$$
\end{lemma}
\begin{proof}
 We need to produce an explicit Markov kernel that allows to approximate the density $\hat{f}_m g$ given an observation from the multinomial model. For all $j=2,\dots,m-1$, let $u_j(x)$ be the (compactly supported) triangular shaped function, such that 
 \begin{equation}\label{eq:u}
  u_j(x_{j-1}^*)=0,\quad u_j(x_{j}^*)=\frac{1}{\nu_0(J_j)}=\frac{m}{\mu_n},\quad u_j(x_{j+1}^*)=0, 
 \end{equation}
linearly interpolated between these values. We also define analogously (compactly supported) trapezoidal shaped functions $u_1$, $u_m$; the former is supported on $[0,x_2^*]$, where it is the linear interpolation of 
$$u_1(0)=u_1(x_1^*)=\frac{1}{\nu_0(J_1)}\quad \textnormal{and}\quad u_1(x_2^*)=0.$$
$u_m$ is defined analogously on $[x_{m-1}^*,1]$ with $u_m(x_{m-1}^*)=0$ and $u_m(x_m^*)=u_m(x)=\frac{1}{\nu_0(J_m)}$, for all $x>x_m^*$. The required (randomized) Markov kernel is then
$$
K\big((k_1,\dots,k_m), A\big) = \int_A u_{X_{(k_1,\dots,k_m)}}(x) \nu_0(dx), \quad \forall (k_1, \dots, k_m) \in \N, \ \sum_i k_i = n,\  A \subset \R,
$$
where $X_{(k_1,\dots,k_m)} \in \{1, \dots, m\}$ is a randomly chosen integer assigning to $j$ the weight $\frac{k_j}{n}$.
\end{proof}

Step 3. Let us denote by $\mathcal N_m$ the statistical model associated with the observation of $m$ independent Gaussian variables $\No(\sqrt{n\gamma_i},\frac{1}{4})$, $i=1,\dots,m$. Since $\frac{\max \gamma_i}{\min \gamma_i}\leq \frac{M}{\kappa}$, one can apply Theorem \ref{teocarter} obtaining
 $$\Delta(\mathcal M_m,\mathcal N_m)=O\Big(\frac{m \ln m}{\sqrt{n}}\Big).$$
 Here the $O$ depends only on $M$ and $\kappa$.
 
 Step 4. Finally, we conclude the proof of Theorem \ref{teo1}, by showing that 
 \begin{equation}\label{eq:1}
\Delta(\mathscr{N}_m,\Wh_n^g)\leq 2\sqrt{n}\sup_{f\in \F} \big(A_m(f)+B_m(f)\big).  
 \end{equation}

 As a preliminary remark note that $\Wh_n^g$ is equivalent to the model that observes a trajectory from:
$$d\bar y_t=\sqrt{f(t)}g(t)dt+\frac{\sqrt{g(t)}}{2\sqrt{n}}dW_t,\quad t\in I.$$
In order to prove \eqref{eq:1} we proceed in the following way: First of all, we prove that $\mathscr N_m$ is equivalent to the model that observes the increments on the intervals $J_i$ of $(\bar y_t)_{t\in I}$. Secondly, we show that the increments of $(\bar y_t)_{t\in I}$ are more informative than another Gaussian process, say $(Y_t^*)_{t\in I}$, that turns out to be very close to $(\bar y_t)_{t\in I}$ in the total variation distance. We then conclude the asymptotic equivalence between $\mathscr{N}_m$ and $\Wh_n^g$ observing that the increments of $(\bar y_t)_{t\in I}$ are obviously less informative than $\Wh_n^g$. 

Let us denote by $\bar Y_j$ the increments of the process $(\bar y_t)$ over the intervals $J_j$,  $j=1,\dots,m$, i.e.
$$\bar Y_j:=\bar y_{v_j}-\bar y_{v_{j-1}}\sim\No\bigg(\int_{J_j}\sqrt{f(y)}\nu_0(dy),\frac{\nu_0(J_j)}{4n}\bigg)$$
and denote by $\mathscr{\bar N}_m$ the statistical model associated with the distributions of these increments. 
As announced we start by bounding the Le Cam distance between $\mathscr{N}_m$ and $\mathscr{\bar N}_m$ showing that 
\begin{equation}\label{eq:normali}
\Delta(\mathscr{N}_m,\mathscr{\bar N}_m)\leq 2\sqrt{n} \sup_{f\in \F} B_m(f), \ \textnormal{ for all m}. 
\end{equation}
In this regard, remark that the experiment $\mathscr{\bar N}_m$ is equivalent to another experiment, say $\mathscr{N}^{\#}_m$, that observes $m$ independent Gaussian random variables of means $\frac{2\sqrt{n}}{\sqrt{\nu_0(J_j)}}\int_{J_j}\sqrt{f(y)}\nu_0(dy)$, $j=1,\dots,m$ and variances identically $1$.
Hence, using also Property \ref{delta0}, Facts \ref{h} and \ref{fact:gaussiane} we get:
\begin{align*}
\Delta(\mathscr{N}_m, \mathscr{\bar N}_m)\leq\Delta(\mathscr{N}_m, \mathscr{N}^{\#}_m)&\leq \sqrt{\sum_{j=1}^m\bigg(\frac{2\sqrt{n}}{\sqrt{\nu_0(J_j)}}\int_{J_j}\sqrt{f(y)}\nu_0(dy)-2\sqrt{n\nu(J_j)}\bigg)^2}.
\end{align*}

Using similar ideas as in Section 8.2 of \cite{cmultinomial} and Lemma 3.2 of \cite{esterlevy}, we introduce a new stochastic process constructed from the random variables $\bar Y_j$'s. To that end recall the notation introduced in the proof of Lemma \ref{lemma:multinomial}, see \eqref{eq:u}, and define
\begin{equation}\label{eq:Y*}
Y_t^*=\sum_{j=1}^m\bar Y_j\int_{I\cap [0,t]} u_j(y)\nu_0(dy)+\frac{1}{2\sqrt{n}}\sum_{j=1}^m\sqrt{\nu_0(J_j)}B_j(t),\quad t\in I,
\end{equation}
where the $(B_j(t))_t$ are independent centered Gaussian processes with variances
$$\textnormal{Var}(B_j(t))=\int_{I\cap [0,t]}u_j(y)\nu_0(dy)-\bigg(\int_{I\cap [0,t]} u_j(y)\nu_0(dy)\bigg)^2.$$
By construction, $(Y_t^*)$ is a Gaussian process with mean and variance given by, respectively:
\begin{align*}
 \E[Y_t^*]&=\sum_{j=1}^m\E[\bar Y_j]\int_{I\cap [0,t]} u_j(y)\nu_0(dy)=\sum_{j=1}^m\bigg(\int_{J_j}\sqrt{f(y)}\nu_0(dy)\bigg)\int_{I\cap [0,t]} u_j(y)\nu_0(dy),\\
 \textnormal{Var}[Y_t^*]&=\sum_{j=1}^m\textnormal{Var}[\bar Y_j]\bigg(\int_{I\cap [0,t]} u_j(y)\nu_0(dy)\bigg)^2+\frac{1}{4 n}\sum_{j=1}^m \nu_0(J_j)\textnormal{Var}(B_j(t))\\
   &= \frac{1}{4 n}\int_{I\cap [0,t]} \sum_{j=1}^m \nu_0(J_j) u_j(y)\nu_0(dy)= \frac{1}{4 n}\int_{I\cap [0,t]} 1 \nu_0(dy)=\frac{\nu_0({I\cap [0,t]})}{4 n}.
\end{align*}
Therefore, 
$$Y^*_t=\int_{I\cap [0,t]} \widehat{\sqrt {f}}_m(y)\nu_0(dy)+\frac{\sqrt{g(t)}}{2\sqrt{n}}W_t,\quad t\in I,$$
where 
$$\widehat{\sqrt {f}}_m(x):=\sum_{j=1}^m\bigg(\int_{J_j}\sqrt{f(y)}\nu_0(dy)\bigg)u_j(x).$$
Applying Fact \ref{fact:processigaussiani}, we get that the total variation distance between the process $(Y_t^*)_{t\in I}$ constructed from the random variables $\bar Y_j$, $j=1,\dots,m$ and the Gaussian process $(Y_t)_{t\in I}$ is bounded by
$$\sqrt{4 n\int_I\big(\widehat{\sqrt {f}}_m(y)-\sqrt{f(y)}\big)^2\nu_0(dy)},$$
as wanted.

\subsection{Proof of Corollary \ref{cor1}}
We start by proving a technical Lemma needed for the proof of Corollary \ref{cor1}. Recall the following notations: $\mu_n=\nu_0(J_j)$, for all $j$ and $\ell_m=\max_{i=1,\dots,m}|v_j-v_{j-1}|$, with the $v_i$'s defined as in Section \ref{sec:parameter}.
\begin{lemma}\label{lemma:ax}
 If $f\in \F_{(\gamma,K,\kappa,M)}$ then
 $$\|f-\hat f_m\|_{L_2(\nu_0)}^2\leq O\Big(\mu_n \ell_n^2+\ell_n^{2+2\gamma}\Big),$$
 with  the $O$ depending on $K, M$ and $\kappa$.
 \end{lemma}
\begin{proof}
Let us consider the Taylor expansion of $f$ at points $x_j^*$, where $x$ denotes a point in $(x_{j-1}^*,x_{j}^*]$ , $j=1,\dots,m$:
\begin{align}\label{eq:taylor}
 f(x)&=f(x_j^*)+f'(x_j^*)(x-x_j^*)+R(x).
\end{align}
The smoothness condition on $f$ allows us to bound the error $R$ as follows: 
 \begin{align*}
  |R(x)|&=\Big|f(x)-f(x_{j}^*)-f'(x_{j}^*)(x-x_{j}^*)\Big|\\
    &=\big|f'(\xi_j)-f'(x_{j}^*)\big||\xi_j-x_{j}^*|\leq K\ell_m^{1+\gamma},
 \end{align*}
where $\xi_j$ is a certain point in $(x_{j-1}^*,x_{j}^*]$.

By the linear character of $\hat f_m$, we can write:
$$\hat f_m(x)=\hat f_m(x_j^*)+\hat f_m'(x_j^*)(x-x_j^*)$$
where $\hat f_m'$ denotes the left or right derivative of $\hat f_{m}$ in $x_j^*$ depending whether $x<x_j^*$ or $x>x_j^*$; this equals to $f'(t)$ for some $t \in J_j$, which allows us to exploit the Hölder condition. Indeed, if $x\in J_j$, $j=1,\dots,m$, then there exists $t\in J_j$ such that: 
\begin{align*}
|f(x)-\hat f_m(x)|&\leq |f(x_j^*)-\hat f_m(x_j^*)|+|f'(x_j^*)-f'(t)||t-x_j^*|+|R(x)|\\
        &\leq |f(x_j^*)- \hat f_m(x_j^*)|+K|t-x_j^*|^{\gamma+1} + K\ell_m^{1+\gamma}\leq |f(x_j^*)- \hat f_m(x_j^*)|+2K\ell_m^{1+\gamma}.
\end{align*}
Using \eqref{eq:taylor} and the fact that $\int_{J_j}(x-x_{j}^*)\nu_0(dx)=0$, one gets:
$$\big|f(x_{j}^*)-\hat f_{m}(x_{j}^*)\big|=\frac{1}{\nu_0(J_j)}\bigg|\int_{J_j}\big(f(x_{j}^*)-f(x)\big)\nu_0(dx)\bigg|\leq K\ell_m^{1+\gamma}.$$

Moreover, observe that, for all $x\in J_i$, $i=1,\dots,m$, $\big|f(x)-\frac{\nu(J_j)}{\nu_0(J_j)}\big|$, is bounded by $3K\ell_m^{1+\gamma}+\ell_m M$, indeed:
 \begin{align*}
  \bigg|f(x)-\frac{\nu(J_j)}{\nu_0(J_j)}\bigg|&=|f(x)-\hat f_m(x_i^*)|\leq |f(x)-\hat f_m(x)|+|\hat f_m(x)-\hat f_m(x_i^*)|\\
                                    & \leq 3K\ell_m^{1+\gamma}+ |\hat f_m'(x_i^*)(x-x_i^*)|
                                    \leq 3K\ell_m^{1+\gamma}+ M \ell_m.
 \end{align*}
 Collecting all the pieces together we find
 \begin{align*}
 \int_I\big(f(x)-\hat f_m(x)\big)^2\nu_0(dx) \leq 2\mu_n\Big(3K\ell_m^{1+\gamma}+M\ell_m\Big)^2+18K^2\ell_m^{2+2\gamma}.
\end{align*}
 \end{proof}
\begin{proof}[Proof of Corollary \ref{cor1}]
First of all, let us observe that $\nu_0(I)$ is finite; indeed, the positivity condition on $f$ ($f(x)\geq \kappa>0$) implies that $\nu_0(I)\leq \frac{1}{\kappa}$.
Also, by means of the fact that $f(x)\geq \kappa$ for all $x\in I$ one can write:
 $$\int_I\bigg(\sqrt{f(x)}-\sqrt{\hat f_m(x)}\bigg)^2 g(x)dx=\int_I\bigg(\frac{f(x)-\hat f_m(x)}{\sqrt{f(x)}+\sqrt{\hat f_m(x)}}\bigg)^2 g(x) dx\leq \frac{1}{4\kappa}\int_I\big(f(x)-\hat f_m(x)\big)^2 g(x)dx.$$
 
 A straightforward application of Lemma \ref{lemma:ax} gives
 $$H_m^2(f)=O\Big(\mu_n \ell_m^2+\ell_m^{2+2\gamma}\Big).$$
 The same bound holds for $A_m^2(f)$ since if $f\in\F_{(\gamma,K,\kappa,M)}$ then $\sqrt f\in\F_{(\gamma,\frac{K}{\sqrt\kappa},\sqrt \kappa,\sqrt M)}$. 
 Moreover, one can see that $B_m$ converges with the same rate as $A_m$. This may be done by explicit computations, see \cite{esterlevy}, Lemma 3.10 for more details.
 \end{proof}

 \appendix
\section{Background}
\subsection{Le Cam theory of statistical experiments}\label{sec:lecam}
A \emph{statistical model} or \emph{experiment} is a triplet $\mo_j=(\X_j,\A_j,\{P_{j,\theta}; \theta\in\Theta\})$ where $\{P_{j,\theta}; \theta\in\Theta\}$ 
is a family of probability distributions all defined on the same $\sigma$-field $\A_j$ over the \emph{sample space} $\X_j$ and $\Theta$ is the \emph{parameter space}.
The \emph{deficiency} $\delta(\mo_1,\mo_2)$ of $\mo_1$
with respect to $\mo_2$ quantifies ``how much information we lose'' by using $\mo_1$ instead of $\mo_2$ and it is defined as
$\delta(\mo_1,\mo_2)=\inf_K\sup_{\theta\in \Theta}||KP_{1,\theta}-P_{2,\theta}||_{TV},$
 where TV stands for ``total variation'' and the infimum is taken over all ``transitions'' $K$ (see \cite{lecam}, page 18). The general definition of transition is quite involved but, for our purposes, it is enough to know that (possibly randomized) Markov kernels are special cases of transitions. By $KP_{1,\theta}$ we mean the image measure of $P_{1,\theta}$ via the Markov kernel $K$, that is
 $$KP_{1,\theta}(A)=\int_{\X_1}K(x,A)P_{1,\theta}(dx),\quad\forall A\in \A_2.$$
 The experiment $K\mo_1=(\X_2,\A_2,\{KP_{1,\theta}; \theta\in\Theta\})$ is called a \emph{randomization} of $\mo_1$ by the Markov kernel $K$. When the kernel $K$ is deterministic, that is $K(x,A)=\I_{A}S(x)$ for some random variable $S:(\X_1,\A_1)\to(\X_2,\A_2)$, the experiment $K\mo_1$ is called the \emph{image experiment by the random variable} $S$.
The Le Cam distance is defined as the symetrization of $\delta$ and it defines a pseudometric. When $\Delta(\mo_1,\mo_2)=0$  the two statistical models are said to be \emph{equivalent}.
Two sequences of statistical models $(\mo_{1}^n)_{n\in\N}$ and $(\mo_{2}^n)_{n\in\N}$ are called \emph{asymptotically equivalent}
if $\Delta(\mo_{1}^n,\mo_{2}^n)$ tends to zero as $n$ goes to infinity. 
A very interesting feature of the $\Delta$-distance is that it can be also translated in terms of statistical decision theory.
Let $\D$ be any (measurable) decision space and let $L:\Theta\times \D\mapsto[0,\infty)$ denote a loss function. Let $\|L\|=\sup_{(\theta,z)\in\Theta\times\D}L(\theta,z)$. Let $\pi_i$ denote a (randomized) decision procedure in the $i$-th experiment. Denote by $R_i(\pi_i,L,\theta)$ the risk from using procedure $\pi_i$ when $L$ is the loss function and $\theta$ is the true value of the parameter. Then, an equivalent definition of the deficiency is:
\begin{align*}
 \delta(\mo_1,\mo_2)=\inf_{\pi_1}\sup_{\pi_2}\sup_{\theta\in\Theta}\sup_{L:\|L\|=1}\big|R_1(\pi_1,L,\theta)-R_2(\pi_2,L,\theta)\big|.
\end{align*}
Thus $\Delta(\mo_1,\mo_2)<\varepsilon$ means that for every procedure $\pi_i$ in problem $i$ there is a procedure $\pi_j$ in problem $j$, $\{i,j\}=\{1,2\}$, with risks differing by at most $\varepsilon$, uniformly over all bounded $L$ and $\theta\in\Theta$.
In particular, when minimax rates of convergence in a nonparametric estimation problem are obtained in one experiment, the same rates automatically hold in any asymptotically equivalent experiment. There is more: When explicit transformations from one experiment to another are obtained, statistical procedures can be carried over from one experiment to the other one.

There are various techniques to bound the Le Cam distance. We report below only the properties that are useful for our purposes. For the proofs see, e.g., \cite{lecam,strasser}.
\begin{property}\label{delta0}
 Let $\mo_j=(\X,\A,\{P_{j,\theta}; \theta\in\Theta\})$, $j=1,2$, be two statistical models having the same sample space and define 
 $\Delta_0(\mo_1,\mo_2):=\sup_{\theta\in\Theta}\|P_{1,\theta}-P_{2,\theta}\|_{TV}.$
 Then, $\Delta(\mo_1,\mo_2)\leq \Delta_0(\mo_1,\mo_2)$.
\end{property}
In particular, Property \ref{delta0} allows us to bound the Le Cam distance between statistical models sharing the same sample space by means of classical bounds for the total variation distance. To that aim, we collect below some useful results.
\begin{fact}\label{h}
 Let $P_1$ and $P_2$ be two probability measures on $\X$, dominated by a common measure $\xi$, with densities $g_{i}=\frac{dP_{i}}{d\xi}$, $i=1,2$. Define
 \begin{align*}
  L_1(P_1,P_2)&=\int_{\X} |g_{1}(x)-g_{2}(x)|\xi(dx), \\
  H(P_1,P_2)&=\bigg(\int_{\X} \Big(\sqrt{g_{1}(x)}-\sqrt{g_{2}(x)}\Big)^2\xi(dx)\bigg)^{1/2}.
 \end{align*}
Then,
\begin{equation*} 
 \|P_1-P_2\|_{TV}=\frac{1}{2}L_1(P_1,P_2)\leq H(P_1,P_2).
\end{equation*}
\end{fact}
\begin{fact}\label{hp}
 Let $P$ and $Q$ be two product measures defined on the same sample space: $P=\otimes_{i=1}^n P_i$, $Q=\otimes_{i=1}^n Q_i$. Then
 \begin{equation*}
  H ^2(P,Q)\leq \sum_{i=1}^nH^2(P_i,Q_i).
 \end{equation*}
\end{fact}
\begin{fact}\label{fact:gaussiane}
Let $Q_1\sim\No(\mu_1,\sigma_1^2)$ and $Q_2\sim\No(\mu_2,\sigma_2^2)$. Then
$$\|Q_1-Q_2\|_{TV}\leq \sqrt{2\bigg(1-\frac{\sigma_1^2}{\sigma_2^2}\bigg)^2+\frac{(\mu_1-\mu_2)^2}{2\sigma_2^2}}.$$
\end{fact}
\begin{fact}\label{fact:processigaussiani}
For $i=1,2$, let $Q_i$, $i=1,2$, be the law on $(C,\Ci)$ of two Gaussian processes of the form
$$X^i_t=\int_{0}^t h_i(s)ds+ \int_0^t \sigma(s)dW_s,\ t\in I$$
where $h_i\in L_2(\R)$ and $\sigma\in\R_{>0}$. Then: 
 $$L_1\big(Q_1,Q_2\big)\leq \sqrt{\int_I\frac{\big(h_1(y)-h_2(y)\big)^2}{\sigma^2(s)}ds}.$$
\end{fact}

\begin{property}\label{fatto3}
 Let $\mo_i=(\X_i,\A_i,\{P_{i,\theta}, \theta\in\Theta\})$, $i=1,2$, be two statistical models. 
Let $S:\X_1\to\X_2$ be a sufficient statistics
such that the distribution of $S$ under $P_{1,\theta}$ is equal to $P_{2,\theta}$. Then $\Delta(\mo_1,\mo_2)=0$. 
\end{property}

Finally, we recall the following result that allows us to bound the Le Cam distance between multinomial and Gaussian variables. According with the notation used throughout the paper, $\M(n,\theta)$ stands for a multinomial distribution of parameters $(n,\theta)$.
\begin{theorem}\label{teocarter}(See \cite{cmultinomial}, Theorem 1 and Sections 7.1, 7.2)
 Let $\mo=\{P_{\theta}:\theta\in\Theta_R\}$, where $P_\theta=\M(n,\theta)$ and $\Theta_R\subset \R^m$ consists of all vectors of probabilities such that
 $$\frac{\max \theta_i}{\min\theta_i}\leq R.$$
 Let $\mathscr Q=\{Q_\theta:\theta\in\Theta_R\}$ where $Q_\theta$ is the multivariate normal distribution with vector mean $(\sqrt{n\theta_1},\dots,\sqrt{n\theta_m})$ and diagonal covariance matrix $\frac{1}{4}I_m$. Then
 $$\Delta(\mo,\mathscr Q)\leq C_R\frac{m\ln m}{\sqrt n}$$
 for a constant $C_R$ that depends only on $R$.
\end{theorem}

\subsection*{Acknowledgements}

I would like to thank my Ph.D supervisor, Sana Louhichi, for several fruitful discussions. I am also very grateful to Markus Reiss for some very insightful exchanges from which the main idea behind this paper emerged.

\printbibliography
\end{document}